\documentclass{amsart}
\usepackage[english]{babel}
\usepackage[latin1]{inputenc}
\usepackage{amsfonts}
\usepackage{amssymb}
\usepackage{amsmath}
\usepackage{amsthm}
\usepackage{graphicx}
\usepackage{epsfig}
\usepackage[all]{xy}

\numberwithin{equation}{section}

\newcommand{\C}{\mathbb{C}}
\newcommand{\R}{\mathbb{R}}

\newtheorem{theorem}{Theorem}[section]
\newtheorem{lemma}[theorem]{Lemma}

\newtheorem{rem}[theorem]{Remark}

\newcommand{\ip}[2]
      {\ensuremath{\langle #1,#2 \rangle}}

\usepackage{amsmath}

\begin{document} \title{ AN ALTERNATE PROOF OF DE BRANGES THEOREM ON  CANONICAL SYSTEMS }\maketitle
\begin{center}\author{KESHAV RAJ ACHARYA \\  \vspace{.2in} Department of Mathematics \\Southern Polytechnic State University\\ Marietta GA\\ kacharya@spsu.edu } \end{center}
\begin{abstract}  The aim of this paper is to show that, in the limit circle case,  the defect index of a symmetric relation induced by  canonical systems (\ref{ca}), is constant on $\C$. This  provides an alternative proof of  De Branges theorem that  the canonical systems (\ref{ca}) with $\operatorname{tr}H \equiv 1$ imply the limit point case. To this end, we  discuss  the spectral theory of a linear relation induced by a canonical system (\ref{ca}).\end{abstract} 

\vspace{0.1in}



\section{ Introduction }
This paper deals with the canonical systems of the following form \begin{equation}\label{ca}Ju'(x)=z H(x)u(x), \,\,\, z \in \C.\end{equation} Here $ J =
\left(\begin{array}{cc}0&-1\\1&0\end{array}\right)$  and $H(x) $  is a $2\times 2 $ positive semi-definite matrix whose entries are locally integrable.  For fixed $z\in \C ,$ a function
$u(.,z):[0,N]\rightarrow \C^2$ is called a solution if $u$ is
absolutely continuous and satisfies (\ref{ca}). Consider the Hilbert space \[ L^2(H,\R_+) = \Big \{f(x,z)
= \left( \begin{array}{c}f_1(x)\\f_2(x)\end{array} \right): \|f\| < \infty \Big \}\] 
provided with an inner product $\big<f, g\big> = \int_{0}^{\infty}f(x)^*
H(x)g(x)dx$. \\

The canonical systems (\ref{ca}) on $L^2(H,\R_+)$ has been studied by Hassi, De Snoo, Winkler, and Remling in  \cite{HW2,cr, RC, HW3} in various context. The Jacobi and Schr\"odinger equations can be written into canonical systems with appropriate choice of $H(x)$.  In addition,  canonical systems  are closely connected with the theory of de Branges spaces and the inverse spectral theory of one dimensional Schr\"odinger operators,  see \cite{RC}. We believe that the extensions of the theories from these equations to the canonical systems is to be of general interest.\\

  If the system  (\ref{ca}) can be written in the form
 \[H(x)^{-1}Ju' = zu\]
  then we may consider this as an eigenvalue equation of an operator on $L^2(H,R_+)$. But $H(x)$ is not invertible in general.  Instead, the system
 (\ref{ca}) induces a linear relation that may have a multi-valued part. Therefore, we  consider this as an eigenvalue problem of a  linear relation induced by (\ref{ca}) on $L^2(H,R_+)$.
 
  For some $z\in \C,$  if the canonical system (\ref{ca}) has all solutions in $ L^2(H,\R_+)$  we say that the system is in the limit circle case, and if the system has unique  solution in $ L^2(H,\R_+)$ we say that  the system is in limit point case. The basic results  in this paper are the following theorems: 
  
\begin{theorem}\label{lc} In the limit circle case, the defect index $\beta ( \mathcal R_0)$ of the symmetric relation $\mathcal R_0,$ induced by (\ref{ca}) is constant on $\C.$ \end{theorem}

 The immediate consequence of the Theorem \ref{lc} is the following theorem

\begin{theorem}[de Branges] \label{db}  The canonical systems with $\operatorname{tr}H \equiv 1$ prevail the limit point case.   \end{theorem}

Theorem \ref{db} has been proved in \cite{db} by function theoretic approach. However the proof was not easily readable to me and we thought of providing an alternate and simple  proof of the theorem.
 
In order to prove the main theorems we use the results from the papers  \cite{HW2,RC,HW3} and use the spectral theory of a linear relation from \cite{ka}.  \\
 
Let $\mathcal{H}$ be a Hilbert space over  $\C$ and denote  by $ \mathcal H^2$ the Hilbert space $ \mathcal {H \oplus H}.$ A linear relation $ \mathcal R = \{(f,g): f,g \in\mathcal{H} \} $ on $\mathcal H$ is a subspace of $ \mathcal H^2.$ The adjoint of $\mathcal R$ on $ \mathcal H$ is a closed linear relation defined by \[ \mathcal R^* = \{ (h,k) \in \mathcal H^2 : \,\,   \ip{g}{h}=\ip{f}{k} \text { for all } (f,g) \in \mathcal R \}.\] A linear relation $ \mathcal S $ is called symmetric if $\mathcal S\subset \mathcal S^*$  and self-adjoint if $\mathcal S = \mathcal S^*$.
  The theory of such relations can  be found in \cite{ka, db,coddington, snoo}. The  regularity domain of $\mathcal R$ is the set \begin{align*}\Gamma(\mathcal R)= \{ z\in \C  :  \exists \,\, C(z)>0 \text{  such that } \|(zf-g)\| \geq C(z) \|f\| \text { for  all  } (f,g) \in \mathcal R\} \end{align*}
  The following theorem has been derived from \cite{ka}.
  
    \begin{theorem}\label{sr1} Suppose $ \mathcal T $ is a self-adjoint relation and suppose $ z\in \Gamma(\mathcal T) $ then \[ \mathcal H = \{ zf-g : \,\,\,\, (f,g) \in  \mathcal T \}.\] \end{theorem}

   The defect index $ \beta(\mathcal R,z)$ is the dimension of defect space \[ R(z-\mathcal R )^{\perp}= \{zf-g: \,\, (f,g) \in  \mathcal R \}^{\perp} \]  It has been shown in \cite{ka} that the defect index $\beta(\mathcal R,z)$  is constant on each connected subset of $\Gamma(\mathcal R)$. Moreover, if $\mathcal R$ is symmetric, then the defect index is constant in the upper and lower half-planes. In addition, it is worth mentioning here the following theorem from \cite{ka} which provides us the condition for a symmetric relation on a Hilbert space to have self-adjoint extension

\begin{theorem} \label{srd}$(1)$ A symmetric relation $\mathcal R $ possess self-adjoint extension if and only if $\beta(\mathcal R,z)$ on lower and upper half-planes are equal.\\
$(2)$ A symmetric extension $\mathcal R'$ of $\mathcal R $ is self-adjoint if and only if $\mathcal R'$ is an $\beta(\mathcal R,z)$-dimensional extension of  $ \mathcal R.$ \end{theorem}

The resolvent set for a closed relation  $\mathcal R$  is the set  \[\rho ( \mathcal R ) = \Big \{z \in \C : \exists \,\,\, T \in B(\mathcal H) \text{ such that }   \mathcal R =  \{ (Tf,zTf-f) : \,\, f\in \mathcal H \}\Big \}\] and the spectrum of $\mathcal R$ is  $\sigma ( \mathcal R )= \C- \rho ( \mathcal R).$\\ 

 We call $S(\mathcal R)= \C- \Gamma(\mathcal R)$ the \emph{spectral kernel} of $ \mathcal R.$ For a self adjoint relation $ \mathcal T $ and $ T = (\mathcal T-z)^{-1}, \,\,z \in \Gamma( \mathcal T).$ The following theorem from \cite{ka}  shows the relation between the spectral kernel and spectrum of a self-adjoint relation.

 \begin{theorem}\label{srp} \noindent

 \begin{enumerate} \item$S(\mathcal T) = \sigma(\mathcal T)$
                  \item If $ \lambda \in S(\mathcal T)$ then $ \frac{1}{z-\lambda}\in S(T)$
                   \item $ S(T) \subset \sigma(T)$ \end{enumerate}
 \end{theorem}
In the next section we discuss about the linear relation induced by a canonical system and prove our main theorems
\section{Relation induced by a Canonical System on $ L^2(H, \R_+)$ and Proof of the Main Theorems} \label{rcs}

 Consider a relation $\mathcal R $ in the Hilbert space $ L^2(H, \R_+)$  induced by (\ref{ca}) as \[ \mathcal R = \{ (f,g) \in \big( L^2(H, \R_+)\big)^2 : \,\, f \in AC, Jf'=Hg\}\] and  is called the maximal relation. This  relation is made up of pairs of equivalence classes $(f,g)$, such that there exists a locally absolutely continuous representative of  $f$ again denoted by $f$, and a representative of $g$, again denoted by $g$, such that $Jf'=Hg \,\, a.e.$ on $\R_+.$ The adjoint relation $ \mathcal R_0 = \mathcal R^*$  is defined by \[  \mathcal R_0 = \{ (f,g) \in \big( L^2(H, \R_+)\big)^2 : \ip{g}{h}= \ip{f}{k} \,\, \text {  for all } (h,k) \in \mathcal R \}\] and is called the minimal relation. It has been shown in \cite{HW2} that $ \mathcal R_0$ is close and symmetric. Moreover, $\mathcal R_0 \subset \mathcal R$ and $(\mathcal R_0)^* = \mathcal R$

\begin{lemma}\cite{HW2}  \label{src2}For each $c,d \in \C^2$ there exists $(\phi_0,\psi_0),(\phi_N,\psi_N) \in \mathcal R$ such that $\phi_0, \phi_N$ have compact support and $\phi(0+)=c, \phi(N-) =d.$ \end{lemma}

\begin{lemma} \cite{HW2} \label{src3} Let $(f,g), (h,k) \in \mathcal R $. Then the following limit exists: \begin{align}\displaystyle \lim_{x\rightarrow\infty}h(x)Jf(x)= h(0+)Jf(0+)-[ \ip{f}{k}-\ip{g}{h}] .\end{align} \end{lemma}

 \begin{lemma}
The minimal relation $\mathcal R_0$ is given by \[\mathcal R_0=\{(f,g) \in \mathcal R : f(0+)=0,\,\, \displaystyle \lim_{x\rightarrow \infty} f^*(x)Jh(x)=0 \text {  for  all } (h,k) \in \mathcal R \} .\] \end{lemma}

 \begin{proof}  By Lemma \ref{src3} we get \[ \{(f,g) \in \mathcal R : f(0+)=0,\,\, \displaystyle \lim_{x\rightarrow \infty} f^*(x)Jh(x)=0 \text {  for  all } (h,k) \in \mathcal R \}\subset \mathcal R_0.\] On the other hand  let $ (f,g)\in R_0 $. By Lemma \ref{src2} for any $u\in \C^2$ there exists $(\phi,\psi) \in \mathcal R$ such that $\phi$ has compact support and $\phi(0+)=c.$  So \begin{align*}  0  & = \ip{f}{\psi}-\ip{g}{\phi} \\ & = \displaystyle  \lim_{x\rightarrow\infty}f^*(x)J\phi(x)- \phi(0+)Jf(0+) \\ & = uJf(0+) .\end{align*} This implies that $f(0+)=0.$   This would  also forces that \[ \displaystyle  \lim_{x\rightarrow\infty}f^*(x)Jh(x)=0 \text{  for all  }  (h,k) \in \mathcal R .\]  \end{proof}
Note that the dimension of the solution space of the system (\ref{ca}) is two. 
\begin{rem} \label{ds} The defect index $\beta(\mathcal R_0)$ of the minimal relation $ \mathcal R_0$ is equal to the number of linearly independent solutions of the system (\ref{ca}) of whose class lie in $ L^2(H, \R_+)$. Therefore, in the limit circle case, the defect indices of  $\mathcal R_0$ are $(2,2).$ \end{rem}

 Since $\mathcal R_0$  has equal defect indices,  by Theorem \ref{srd} it has  self-adjoint extensions say $\mathcal T.$  Consider a relation ,\begin{align*}  \mathcal T^{\alpha,\beta}= \{ (f,g)\in \mathcal R : f_1(0)\sin\alpha +f_2(0)\cos\alpha =0,\,\,\,\,\\ f_1(N)\sin\beta +f_2(N)\cos\beta =0, \,\,\,\, \alpha, \beta \in (0,\pi]\} \end{align*} on a compact interval  $[0,N].$ 

  \begin{lemma}$ \mathcal T^{\alpha,\beta}$ is a self-adjoint relation.\end{lemma}
  \begin{proof}

  Clearly $\mathcal T^{\alpha,\beta}$  is a symmetric relation because of the boundary conditions at $0$ and $N$. We will show that $ \mathcal T^{\alpha,\beta}$ is an $2$-dimensional extension of $ \mathcal R_0.$ Then by Theorem \ref{srd}, $ \mathcal T^{\alpha,\beta}$ is an self-adjoint relation.   By Lemma \ref{src2}, for $  c = \begin{pmatrix}-\cos\alpha\\ \sin\alpha \end{pmatrix}$ and $ w = \begin{pmatrix} -\cos\beta\\ \sin\beta \end{pmatrix} \in \C^2$ there exists $ \phi_0$ and $\phi_N $ in $ D(\mathcal R)$ such that $ \phi_0(0+)= c$ and $\phi_N(N-)= w $ and the support of $ \phi_0$ and $ \phi_N$ are contained in $ [0,N].$ Clearly $ \phi_0$ , $ \phi_N$ are  linearly independent  but $ \phi_0$ , $ \phi_N$ are not in $D(\mathcal R_0).$ This shows that $ D(\mathcal R_0)\subset D(R_0)+ L(\phi_0,\phi_N)\subset D(\mathcal T^{\alpha,\beta}).$  Because of the boundary condition at $0$ and $N$, $D(\mathcal T^{\alpha,\beta})$ is $2$-dimensional restriction of $D(\mathcal R).$ Hence $D(\mathcal T^{\alpha,\beta})= D(\mathcal R_0)+ L(\phi_0,\phi_N).$ Therefore, $\mathcal T^{\alpha,\beta} $ is $2$-dimensional extension of  $\mathcal R_0$ so that $\mathcal T^{\alpha,\beta} $ is a self-adjoint relation. \end{proof}

  Let $u(x,z)$ and $v(x,z)$ be the solution of the canonical system (\ref{ca}) on $ [0,N]$ with the initial values \[  u(0,z)=\begin{pmatrix}1\\0\end{pmatrix}, v(0,z)=\begin{pmatrix}0\\ 1\end{pmatrix} .\]  For $z\in C^+$ there is a unique $m(z)$ such that $ f(x,z)= u(x,z)+ m(z)v(x,z)$ satisfying \[f_{1}(N,z) \sin \beta + f_{2}(N,z) \cos\beta = 0 .\] This is well defined because $u $ does not satisfy the boundary condition at $N$ otherwise $ z$ will be an eigenvalue of some self-adjoint relation $  \mathcal T^{\alpha,\beta}.$ 

  Next, we describe the spectrum of $\mathcal T^{\alpha,\beta} .$ Let \[ T(x,z)=\begin{pmatrix} u_{ 1}(x,z)& v_{1}(x,z) \\                                                                                                                      u_{ 2}(x,z)& v_{ 2}(x,z) \\  \end{pmatrix},\,\,\,\, T(0,z) = \begin{pmatrix} 1 & 0 \\  0 & 1   \end{pmatrix}\]

and define \[ w_{\alpha}(x,z)= \frac{1}{\sin\alpha+ m(z)\cos\alpha } T(x,z)\begin{pmatrix}
                                                                                          \cos\alpha \\
                                                                                          -\sin\alpha
                                                                                        \end{pmatrix}
 \] It is not hard to see that

\begin{lemma} \label{ir}Using the notation above we have \[ f(x,z)w_{\alpha}(x,\bar{z})^*- w_{\alpha}(x,z)f(x,\bar{z})^* = T(x,z)JT(x,\bar{z})^* = J.\] \end{lemma}

 \begin{lemma} Let $z\in \Gamma(\mathcal T^{\alpha,\beta})$ then $ (\mathcal T^{\alpha,\beta}-z)^{-1}$ is a bounded linear operator and is defined by \[(\mathcal T^{\alpha,\beta}-z)^{-1}h(x) = \int_0^N G(x,t,z)H(t)h(t)dt,\] where $G(x,t,z)= \left \{ \begin{array}{cc}
f(x,z)w_{\alpha}(t,\bar{z_0})^*   & \text{  if }  0<t\leq x \\
w_{\alpha}(t,\bar{z})f(x,\bar{z_0})  & \text{  if }    x<t\leq N.
 \end{array} \right .$\end{lemma}

  \begin{proof} Let $ y(x,z)= \int_0^N G(x,t,z)H(t)h(t)dt.$ We show that $y(x,z)$ solves the inhomogeneous equation \[ Jy'= zHy - Hh \] for $a.e. x>0.$
 Here \[y(x,z)= \int _0^x f(x,z)w_{\alpha}(t,\bar{z})^* H(t)h(t)dt + \int_x^N w_{\alpha}(x,z)f(t,\bar{z})^*H(t)h(t)dt \] and $ Jf'=zHf, Jw_{\alpha}'= zH w_{\alpha}.$ Then on differentiation we get,
  \begin{align*}y'(x,z) =  f(x,z)w_{\alpha}(x,\bar{z})^* H(x)h(x) + f'(x,z)\int _0^x w_{\alpha}(t,\bar{z})^* H(t)h(t)dt \\ - w_{\alpha}(x,z)f(x,\bar{z})^*H(x)h(x) + w_{\alpha}'(x,z)\int_x^N f(t,\bar{z})^*H(t)h(t)dt .\end{align*} Then \begin{align*}J y'(x,z) =  Jf(x,z)w_{\alpha}(x,\bar{z})^* H(x)h(x) + Jf'(x,z)\int _0^x w_{\alpha}(t,\bar{z})^* H(t)h(t)dt \\ - J w_{\alpha}(x,z)f(x,\bar{z})^*H(x)h(x) + Jw_{\alpha}'(x,z)\int_x^N f(t,\bar{z})^*H(t)h(t)dt \\ =  Jf(x,z)w_{\alpha}(x,\bar{z})^* H(x)h(x) + zH f(x,z)\int _0^x w_{\alpha}(t,\bar{z})^* H(t)h(t)dt \\ - Jw_{\alpha}(x,z)f(x,\bar{z})^*H(x)h(x) + zHw_{\alpha}(x,z)\int_x^N f(t,\bar{z})^*H(t)h(t)dt . \\ =  J \Big( f(x,z)w_{\alpha}(x,\bar{z})^* -w_{\alpha}(x,z)f(x,\bar{z})^* \Big) H h + \hspace{1.3in}\\  zH \Big( \int _0^x f(x,z) w_{\alpha}(t,\bar{z})^* H(t)h(t)dt  + \int_x^N w_{\alpha}(x,z) f(t,\bar{z})^*H(t)h(t)dt\Big) \\ = JJ Hh + zHy \hspace{3.1in}\\ = zHy-Hh. \hspace{3.28in}\end{align*}

  On the other hand denote $g(x,z)$ as \[g(x,z)= (\mathcal T^{ \alpha,\beta}-z)^{-1}h(x)\] then by Theorem \ref{sr1}, $h(x)=zu-v $ for some $ (u,v) \in \mathcal T^{\alpha,\beta} $  so that  $(g,zg-h) \in \mathcal T^{\alpha,\beta}. $ So $g(x,z)$ also satisfies the inhomogeneous problem and $ g(x,z)\in D(\mathcal T^{\alpha,\beta}),$ it satisfies the boundary condition which implies that $g(0,z)=\begin{pmatrix}
  \cos\alpha \\ -\sin\alpha\end{pmatrix}c(z)$ for some scalar $c(z).$
 We have \[y(0,z)= \frac{1}{\sin\alpha+  m(z) \cos\alpha}\begin{pmatrix}
  \cos\alpha \\ -\sin\alpha\end{pmatrix} \ip{f(x,\bar{z})}{h}\] Now \begin{align*}  \ip{ f(.,\bar{z})}{h} &= \ip{f(.,\bar{z})}{h} - \ip{f(.,\bar{z})}{zg} + \ip{f(.,\bar{z})}{zg}  \\ &= \ip{f(.,\bar{z})}{h -zg}+ z\ip{f(.,\bar{z})}{g}\\& =
  - \int_0^N f(x,\bar{z})^* H (zg-h)dx + z \int_0^N f(x,\bar{z})^* H gdx \\&= - \int_0^N f(x,\bar{z})^* Jg'dx -\int_0^N f'(x,\bar{z})^* J gdx \\ & = f(0,\bar{z})^* Jg(0,z)- f(N,\bar{z})^* Jg(N,z). \end{align*} Since both $ f(x,z)$ and $g(x,z)$ satisfies the same  boundary condition at $N$ $f(N,\bar{z})^* Jg(N,z)=0.$ Now  \begin{align*} f(0,\bar{z})^* Jg(0,z) = (1, m(z)) \begin{pmatrix} 0 & -1 \\ 1 & 0  \end{pmatrix} \begin{pmatrix} \cos\alpha \\ -\sin\alpha  \end{pmatrix} c(z).\end{align*}So  \begin{align*} y(0,z) & = \frac{1}{\sin\alpha+  m(z) \cos\alpha}\begin{pmatrix}
  \cos\alpha \\ -\sin\alpha\end{pmatrix}(1, m(z)) \begin{pmatrix} 0 & -1 \\ 1 & 0  \end{pmatrix} \begin{pmatrix} \cos\alpha \\ -\sin\alpha  \end{pmatrix} c(z)  \\ & = \frac{1}{\sin\alpha+  m(z) \cos\alpha} \begin{pmatrix} \cos\alpha \\ -\sin\alpha  \end{pmatrix}( m(z), -1 ) \begin{pmatrix} \cos\alpha \\ -\sin\alpha  \end{pmatrix} c(z)\\ & = \begin{pmatrix} \cos\alpha \\ -\sin\alpha  \end{pmatrix} c(z)\\ & = g(0,z). \end{align*}
  By uniqueness we must have , $g(x,z)=y(x,z).$ Moreover, $ (\mathcal T^{\alpha,\beta}-z)^{-1}$ is a bounded linear operator.

  \end{proof}

   Now define a map $V: L^2(H,[0,N])\rightarrow L^2(I,[0,N])$ by \[ Vy= H^{\frac{1}{2}}(x)y(x).\] Here $H^{\frac{1}{2}}(x)$  is the unique positive semi-definite square root of $H(x).$  Then $V$ is an isometry and hence maps $L^2(H,[0,N])$ unitarily onto the range $R(V)\subset L^2(I,[0,N]).$ Define an integral operator $\mathcal L$ on $ L^2(I,[0,N])$ as \[ (\mathcal L f)(x)= \int _0^N L(x,t)f(t)dt, \,\,\,\,L(x,t)= H^{\frac{1}{2}}(x) G(x,t)H^{\frac{1}{2}}(t). \] The kernel $L$ is square integrable since
\begin{align*} \int_0^N\int_0^N \|L^*L\|dxdt \leq\int_0^N\int_0^N \|VG^*\|\|(VG)^*\|dxdt\\ \leq \int_0^N\int_0^N\|G^*\| \|G\|dxdt <\infty. \end{align*} So $\mathcal L$ is a Hilbert-Schmidt operator and thus compact. Since $L(x,t)= L^*(t,x), \mathcal L$ is also self-adjoint.

\begin{lemma} \cite{RC} \label{lem6.3}  Let $f\in L^2(I,[0,N]), \lambda \neq 0,  $ then the following statements are equivalent:
\begin{enumerate}
  \item $ \mathcal L f = \lambda^{-1}f. $
  \item $f \in R(V),$ and the unique $ y \in L^2(H,[0,N])$ with $ Vy=f $ solves $( \mathcal T^{\alpha,\beta}-z )^{-1}y = \lambda y .$
      \end{enumerate}
 \end{lemma}

\begin{proof}  For all $g \in L^2(I,[0,N]) $ we have, \[ (\mathcal Lg)(x)= H^{\frac{1}{2}}(x)w(x) \text{  where } w(x)= \int_0^N G(x,t)H^{\frac{1}{2}}(t)g(t)dt,\] lies in $L^2(H,[0,N]).$ Then $R(\mathcal L)\subset R (V).$ Now if $(1)$ holds then $f= \lambda \mathcal L f \in R(V).$ So $f=V(y)$ for unique $ y \in L^2(H,[0,N])$ and \[f(x)= H^{\frac{1}{2}}(x)y(x) = \lambda \mathcal L y(x)= \lambda H^{\frac{1}{2}}(x) \int_0^N G(x,t) H(t)y(t)dt\] for $a.e. x\in [0,N].$ In other words, \[ H^{\frac{1}{2}}(x)\big(y(x)-\lambda \int_0^N G(x,t) H(t)y(t)dt \big)=0.\] Conversely if $(2)$ holds,\[ \lambda y=  \int_0^N G(x,t) H(t)y(t)dt\] then $  H^{\frac{1}{2}}(x)y= \frac{1}{\lambda} \int_0^N H^{\frac{1}{2}}(x)G(x,t)H(t)y(t)dt.$

\end{proof}

\begin{lemma} \label{lem6.4} Let $z\in \C$. For any $ \lambda \neq z $,  if $(f, \lambda f) \in \mathcal T^{\alpha, \beta}$ then  $f$ solves $( \mathcal T^{\alpha,\beta}-z)^{-1}y= \frac{1}{\lambda -z}y .$  Conversely, if  $y \in L^2(H,[0,N]) $ and $y$ solves $(T^{\alpha,\beta}-z)^{-1}y= \lambda y$ then $(y,(z+\frac{1}{\lambda})y) \in \mathcal T^{\alpha, \beta}$.  \end{lemma}

\begin{proof} Let $(f, \lambda f) \in \mathcal T^{\alpha, \beta}$ then $ (f, \lambda f -zf) \in (\mathcal T^{\alpha, \beta}-z) .$ It follows that   \[((\lambda  -z)f,f) \in (\mathcal T^{\alpha, \beta}-z)^{-1} \Rightarrow  \Big(f,\frac{1}{(\lambda  -z)} \Big)\in (\mathcal T^{\alpha, \beta}-z)^{-1}.\] This means that $f$ solves \[( \mathcal T^{\alpha,\beta}-z)^{-1}y= \frac{1}{\lambda -z}y .\] Conversely suppose $y \in L^2(H,[0,N]) $ and $y$ solves \[(\mathcal T^{\alpha,\beta}-z)^{-1}y= \lambda y.\] That is $ (y,\lambda y) \in (\mathcal T^{\alpha,\beta}-z)^{-1}$ so that $(\lambda y ,y ) \in (\mathcal T^{\alpha,\beta}-z) .$ So there is $ (f,g) \in \mathcal T^{\alpha,\beta}$ such that $ \lambda y =f$ and \[g-zf =y   \Rightarrow g= y+z \lambda y . \] Hence \[(\lambda y ,y+z \lambda y ) \in \mathcal T^{\alpha,\beta} \Rightarrow  \Big( y , (z+\frac{1}{ \lambda} y )\Big) \in \mathcal T^{\alpha,\beta} .\]\end{proof}

By Lemma \ref{lem6.3} we see that there is a one to one correspondence of eigenvalues (eigenfunctions) for the operator $\mathcal L$ and $ (\mathcal T^{\alpha,\beta}-z)^{-1}.$ As $\mathcal L$ is compact operator, it has only discrete spectrum consisting of only eigenvalues. Since $(\mathcal T^{\alpha,\beta}-z)^{-1}$ is unitarily equivalent with $ \mathcal L\downharpoonright_{ R(V)}$, that is \[ V^{-1}\mathcal L\downharpoonright_{R(V)}V =(\mathcal T^{\alpha,\beta}-z)^{-1},\] $(\mathcal T^{\alpha,\beta}-z)^{-1}$  has only discrete spectrum consisting of only eigenvalues. Then by Theorem \ref{srp}, $\mathcal T^{\alpha,\beta}$
  has only discrete spectrum. By
Lemma\ref {lem6.4}, the spectrum of $\mathcal T^{\alpha,\beta}$ consists only eigenvalues.  Hence we have \[ \sigma(\mathcal T^{\alpha,\beta}) = \{ z\in \C: u_{\alpha 1}(N,z)\sin\beta+ u_{\alpha 2}\cos\beta = 0\}.\]

We would like to extend this idea over the half line $\R_+.$ First  note that we are considering the limit circle case of the system (\ref{ca}). That implies  for any $z\in \C^+$ the defect indices of $\mathcal R_0 $  are $(2,2)$.  Suppose $ p \in D(\mathcal R) \smallsetminus D(\mathcal R_0)$ such that $\displaystyle \lim_{x\rightarrow \infty } p(x)^* J p(x)= 0 .$ Such function clearly exists.

  Consider the relation \begin{align*}\mathcal T^{\alpha,p}  = \{ (f,g) \in \mathcal R : f_1(0)\sin\alpha+ f_2(0,z)\cos\alpha = 0 \\ \text{  and  } \displaystyle \lim_{x\rightarrow \infty } f(x)^* J p(x)= 0 \}.\end{align*}

   \begin{lemma}$\mathcal T^{\alpha,p}$ defines a self-adjoint extension of $ \mathcal R_{0}.$ \end{lemma}

   \begin{proof} Let $u(x)$ be a solution of the system (\ref{ca}) with some initial values and $p(x)$ as above. Define $u_0(x)=0 $ near the neighborhood of $\infty$ ie $\displaystyle \lim_{x\rightarrow\infty}u_0(x)=0$ and $u_0(x)=u(x)$ otherwise. Similarly, $p_0(x)=0$ in the neighborhood of $0$ and $p_0(x)=p(x)$ otherwise. Then clearly $u_0,p_0 \notin D(\mathcal R_{0}) $ and are linearly independent. Clearly $ D(\mathcal R_{0})+L(u_0,p_0)\subset D(\mathcal T^{\alpha,p}) . $ Since  $D(\mathcal T^{\alpha,p})$ is  at least $2$-dimensional restriction of $D(\mathcal R)$,  \[D(\mathcal T^{\alpha,p})= D(\mathcal R_{0})+L(u_0,p_0).\] Hence $\mathcal T^{\alpha,p}$ is a self-adjoint relation.  \end{proof}

      We next discuss the spectrum of $\mathcal T^{\alpha,p}.$
     Let $u(x,z)$ and $v(x,z)$ be two linearly independent solutions of the system (\ref{ca}) with\[  u(0,z)=\begin{pmatrix}1\\0\end{pmatrix}, v(0,z)=\begin{pmatrix}0\\ 1\end{pmatrix} .\]
      Let $z \in \C^+$ and as above write $f(x,z)= u(x,z)+m(z)v(x,z) \in L^2(H,\R_+)$ satisfying $\displaystyle \lim_{x\rightarrow \infty}f(x,z)^*JP(x)=0.$  Let $T(x,z)= \begin{pmatrix}u_{ 1} & v_{1}\\u_{ 2} & v_{ 2} \end{pmatrix}$ and \[w_{\alpha}(x,z)= \frac{1}{   \sin\alpha + m(z)\cos\alpha } T(x,z)\begin{pmatrix}\cos\alpha\\ -\sin\alpha\end{pmatrix}.\] Then as  in Lemma \ref{ir} we have, \[f(x,z)w_{\alpha}(x,\bar{z})^*-w_{\alpha}(x,z)f(x,\bar{z})^* = T(x,z)JT (x,\bar{z})^*= J.\]

      \begin{lemma} Let $z\in \rho(\mathcal T^{\alpha,p})$ then the resolvent operator $ (\mathcal T^{\alpha,p}-z)^{-1}$ is given by \[(\mathcal T^{\alpha,p}-z)^{-1}h(x)= \int _0^{\infty}G(x,t,z)H(t)h(t)dt \]  where  $ G(x,t,z)= \left \{ \begin{array}{cc}
f(x,z)w_{\alpha}(t,\bar{z})^*   & \text{  if }  0<t\leq x \\
w_{\alpha}(t,\bar{z})f(x,\bar{z})  & \text{  if }    x<t\leq \infty \end{array} \right.$ \end{lemma}

 \begin{proof}Let $y(x,z)= \int_0^{\infty}G(x,t,z)H(t)h(t)dt$ then $y$ solves the inhomogeneous equation \[ Jy'= zHy-Hh .\] This clear by differentiating \[y(x,z)= \int _0^x f(x,z)w_{\alpha}(t,\bar{z})^* H(t)h(t)dt + \int_x^{\infty} w_{\alpha}(x,z)f(t,\bar{z})^*H(t)h(t)dt. \] On the other hand denote $g(x,z)$ as $g(x,z)= (\mathcal T^{ \alpha,p}-z)^{-1}h(x)$ then by Theorem \ref{sr1}, $h(x)=zu-v $ for some $ (u,v) \in \mathcal T^{\alpha,p} $  so that  $(g,zg-h) \in \mathcal T^{\alpha,p}$ and hence $g$ satisfies the inhomogeneous equation. Since $g\in D(\mathcal T^{\alpha,p})$ \begin{align*}g_1(0,z)\sin\alpha + g_2(0,z)\cos\alpha =0,\displaystyle \lim_{x\rightarrow \infty} g^*(x,z)Jp(x,z)= 0.\end{align*} We also have $ \displaystyle \lim_{x\rightarrow \infty}f^*(x,z)Jg(x,z)=0 $  and  $ g(0,z)= \begin{pmatrix}\cos\alpha\\ -\sin\alpha\end{pmatrix}c(z)$ for some
scalar $ c(z).$ But also we have \[y(0,z) = \frac{1}{ (m(z) \cos \alpha +\sin\alpha)} \begin{pmatrix}\cos\alpha\\ \sin\alpha\end{pmatrix} \ip{f(\bar{z})}{h}.\] Here \begin{align*}  \ip{f(\bar{z})}{h} &=  \ip{f(\bar{z})}{h}- \ip{f(\bar{z})}{zg} - \ip{\bar{z}f(\bar{z})}{g}\\ &= \ip{f(\bar{z})}{h+zg}-\ip{\bar{z}f(\bar{z})}{g} \\ &= f ^*(0,\bar{z})J g(0,z)- \displaystyle \lim_{x\rightarrow \infty}f^*(x,z)Jg(x,z)\\ &= f ^*(0,\bar{z})J g(0,z).\end{align*} Hence $y(0,z)= g(0,z).$ By uniqueness we have $ y(x,z)= g(x,z).$ \end{proof}

Now define a map $V: L^2(H, \R_+)\rightarrow L^2(I, \R_+)$ by $ Vy= H^{\frac{1}{2}}(x)y(x). $  $V$ is isometry and maps unitarily onto the range $R(V)\subset L^2(I, \R_+) .$

 Define an integral operator $\mathcal L$ on $L^2(I, \R_+)$ by \[ (\mathcal L g)(x) = \int _0^{\infty} L(x,t)g(t)dt,\,\,\, L(x,t)= H^{\frac{1}{2}}(x)G(x,t,z)H^{\frac{1}{2}}(t).\]
 Then as before the kernel  $L$ is square integrable. This means that \[ \int _0^{\infty}\int _0^{\infty}\parallel L^*L\parallel < \infty .\] Hence $\mathcal L$ is a Hilbert Schmidt  a operator and so is  a compact operator.
The following two lemmas are extended from the bounded interval $[0,N]$ to $\R_+$ and the proofs are exactly the same as the proofs of Lemma \ref{lem6.3} and Lemma \ref{lem6.4}.

\begin{lemma}\label{lem6.5} \cite{RC}   Let $f\in L^2(I, R_+), \lambda \neq 0,  $ then the following statements are equivalent:
\begin{enumerate}
  \item $ \mathcal L f = \lambda^{-1}f. $
  \item $f \in R(V),$ and the unique $ y \in L^2(H,R_+)$ with $ Vy=f $ solves $( \mathcal T^{\alpha,p}-z )^{-1}y = \lambda y .$
      \end{enumerate}
 \end{lemma}

\begin{lemma}\label{lem6.6} Let $z\in \C$. For any $ \lambda \neq z $,  if $(y, \lambda y) \in \mathcal T^{\alpha, p}$ then  $y$ solves $( \mathcal T^{\alpha,p}-z)^{-1}y= \frac{1}{\lambda -z}y .$  Conversely, if  $y \in L^2(H, R_+) $ and $y$ solves $(T^{\alpha,p}-z)^{-1}y= \lambda y$ then $(y,(z+\frac{1}{\lambda})y) \in \mathcal T^{\alpha,p}$.  \end{lemma}

 Again by Lemma \ref{lem6.5},  we have a one to one correspondence of eigenvalues (eigenfunctions) for the operator $\mathcal L$ and $ (\mathcal T^{\alpha,p}-z)^{-1}.$ As $\mathcal L$ is compact operator, it has only discrete spectrum consisting of only eigenvalues and possibly zero.  Since $(\mathcal T^{\alpha,p}-z)^{-1}$ is unitarily equivalent with $ \mathcal L\downharpoonright_{ R(V)}$, that is $ V^{-1}\mathcal L\downharpoonright_{R(V)}V =(\mathcal T^{\alpha,p}-z)^{-1}, (\mathcal T^{\alpha,p}-z)^{-1} $ has only discrete spectrum consisting of only eigenvalues. Then by Theorem \ref{srp}, $\mathcal T^{\alpha,p}$
  has only discrete spectrum. By Lemma \ref{lem6.6} the spectrum of $\mathcal T^{\alpha,p}$ consists of only eigenvalues.

With these theory in hand, we are now ready to prove the main theorems

\begin{proof}[Proof of theorem \ref{lc}]

 Since $ \mathcal R_0$ is a symmetric relation,  the defect index $\beta(\mathcal R_0,z)$ is constant on upper and lower half planes.  In the limit-circle case, if $z$ is in upper or lower half-planes, $ \beta(\mathcal R_0,z) = 2$. Suppose  $\beta(\mathcal R_0,\lambda) < 2 $ for some $ \lambda \in \R.$ Since $\Gamma(\mathcal R_0)$ is open, $\lambda \notin \Gamma(\mathcal R_0)$ and hence $ \lambda \in S(\mathcal R_0).$  Since for each $ \alpha \in (0,\pi]$,  $\mathcal T^{\alpha,p}$  is self-adjoint extension of $\mathcal R_0$,  $ \lambda \in S(\mathcal T^{\alpha,p}) =  \sigma(\mathcal T^{\alpha,p})$. In the limit-circle case, $\sigma(\mathcal T^{\alpha,p})$ consists of  only eigenvalues. Therefore, $\lambda$ is an eigenvalue for all  boundary conditions $\alpha$ at $0$. However, this is impossible unless $\beta(\mathcal R_0, \lambda) = 2 $. \end{proof}

\begin{proof}[Proof of theorem \ref{db} ] Suppose it prevails the limit-circle case.  By Theorem \ref{lc}, the defect index  $\beta(\mathcal R_0,z)=  \dim N(\mathcal R,\bar{z}) =2$ for all $z\in \C$. In other words, for any $z\in \C $, all  solutions of (\ref{ca}) are in $L^2(H,\R_+)$. In particular, the constant solutions $ u(x)= \begin{pmatrix} 1 \\ 0 \end{pmatrix}$  and $ v(x)= \begin{pmatrix}0 \\1 \end{pmatrix}$  of (\ref{ca}) when $z=0$, are in $L^2(H,\R_+)$. But this is not possible because $ \displaystyle  \int_0^{\infty}u(x)^* H(x)u(x)dx +  \int_0^{\infty}v(x)^* H(x)v(x)dx = \int_0^{\infty}\operatorname{tr} H(x)dx = \infty.$  \end{proof}

\vspace{0.5in}

\providecommand{\bysame}{\leavevmode\hbox to3em{\hrulefill}\thinspace}
\providecommand{\MR}{\relax\ifhmode\unskip\space\fi MR }
\providecommand{\MRhref}[2]{%
  \href{http://www.ams.org/mathscinet-getitem?mr=#1}{#2}
}
\providecommand{\href}[2]{#2}

\vspace{0.5 in}



\end{document}